\documentclass[reqno,10pt]{amsart}
\usepackage{amscd,amssymb}
\usepackage[arrow,matrix]{xy}
\usepackage{graphicx}
\usepackage{color}
\oddsidemargin=0.3in \evensidemargin=0.3in \theoremstyle{plain}
\newtheorem{theorem}[subsection]{Theorem}
\newtheorem{lemma}[subsection]{Lemma}
\newtheorem{proposition}[subsection]{Proposition}

\newtheorem{corollary}[subsection]{Corollary}
\theoremstyle{definition}
\newtheorem{remark}[subsection]{Remark}
\newtheorem{definition}[subsection]{Definition}
\newtheorem{example}[subsection]{Example}

\numberwithin{equation}{section}

\title[Positivity of sums and integrals for $\nabla$-convex and Completely monotonic functions]{Positivity of sums
and integrals for higher order $\nabla-$convex and completely monotonic functions}
\author{Muhammad Adnan$^{1}$}
\address{1-Department of Mathematics, University of Karachi, University Road, Karachi-75270, Pakistan}
\email{adnanalfa84@yahoo.com}
\author{Asif R. Khan$^{1}$}
\email{asifrk@uok.edu.pk}
\author{Faraz Mehmood$^{1,2}$}
\address{2-Dawood University of Engineering and Technology, Karachi, Pakistan}
\email{faraz.mehmood@duet.edu.pk}
\email{farazmehmood@yahoo.com}
\date{April 10, 2017}
 \dedicatory{ } 
\begin{document}
\thanks{\textbf{Acknowledgement:} We are very much thankful to Prof. J. E.~Pe\v cari\' c and Prof. Sanja Varosane\'c for their useful comments and suggestion for this article.}

\begin{abstract}
We extend the definitions of $\nabla-$convex and completely monotonic functions for two variables. Some general identities of Popoviciu type for sum $\sum \sum p_{ij} f(y_i, z_j)$ and integrals $\int P(y)f(y) dy$, $\int \int P(y,z) f(y,z) dy \, dz$
are deduced. Using obtained identities, positivity of these expressions are characterized for  higher order $\nabla-$convex and completely monotonic functions. Some applications in terms of generalized Cauchy means and exponential convexity are given.
\end{abstract}
\maketitle
\section{\bf{Introduction}}
 Over past few decades the notion of completely monotonic functions has gained much popularity among researchers in analysis and other related fields due to their interesting properties (see  \cite{[16]}) and higher applicability (see \cite{[6]}). As it is evident from the following lines taken from a paper with the title ``Completely monotone functions: a digest'' \cite{[CMF]} written by Milan Merkle. He writes ``A brief search in MathSciNet reveals total of 286 items that mention this class of functions in the title from 1932 till the end of the year 2011; 98 of them have been published since the beginning of 2006''. We would like to obtain discrete identity and inequality of two variables and we would like to state general integral identities and inequalities respectively for higher order differentiable function of one variable and two variables. These identities and inequalities would be generalization of several established results. We would also discuss the characterization of Popoviciu$-$type  positivity of these general integrals involving $\nabla-$convex and completely monotonic functions. We would give new generalized, mean value theorems of Lagrange and Cauchy$-$type as well, and also discuss its exponential convexity with the help of various examples.

Let us recall, few useable definitions and significant results regarding the convex functions extracted from \cite{redbook} (see also \cite{Asif-popoviciu}). Throughout the article $I$ an interval in $\mathbb{R}$. Also throughout the article we would use the following notations for some subsets of $\mathbb{R}$, $\mathbb{R_{*}}=[0,\infty)$ and $\mathbb{R_{+}}=(0,\infty)$.
\begin{definition}
The \emph{$m-$divided difference} of  a function $f:I\to \mathbb{R}$, at different points $y_i,y_{i+1},\ldots \\,y_{i+m}\in I=[a,b]\subset \mathbb{R}$, where $i\in \mathbb{N}$ is stated as:
\begin{eqnarray*}
 \left[y_j;f\right]&=&f\left(y_j\right),\quad j\in\{i,i+1,\ldots, i+m\} \\
 \left[y_i,\ldots ,y_{i+m};f\right]&=&\frac{\left[y_{i+1},\ldots ,y_{i+m};f\right]-\left[y_i,\ldots  ,y_{i+m-1};f\right]}
{y_{i+m}-y_i}.
\end{eqnarray*}
\end{definition}
\begin{remark}
Let us denote
\([y_i,y_{i+1},\cdots ,y_{i+m};f]\) by $ \Delta_{(m)} f(y_i).$ The value $\left[y_i,\ldots ,y_{i+m};f \right]$ is independent of order of points $y_i,y_{i+1},\ldots ,y_{i+m}$. This definition can extended by including the cases for more than one points coincide by applying the respective limits.
\end{remark}
\begin{definition}
A function $f:I\rightarrow \mathbb{R} $, is called a \emph{$m-$convex} or \emph{$m$th order convex} , if the inequality $\Delta_{(m)}f(y_i)\geq 0$ holds $\forall (m+1)$ different points $y_i,\ldots ,y_{i+m} \in I$.

Further that if $m$th order derivative of function exists, then function is convex of order $m$ if and only if $f^{(m)}\geq 0$.
\end{definition}
\begin{definition}
A function $f:I\rightarrow \mathbb{R} $, is said to be $m-\nabla-$convex or \emph{$\nabla-$convex of order $m$}, if $\forall (m+1)$ different points $y_i,y_{i+1},\ldots ,y_{i+m}$ we have
$\nabla_{(m)} f(y_i)=(-1)^m \Delta_{(m)}f(y_i) \geq 0$.

Further that if $m$th order derivative of function exists, then function is $\nabla-$convex of order $m$ if and only if $(-1)^{m}f^{(m)}\geq 0$.
\end{definition}
\begin{definition}
Let $E =\{y_{1},y_{2},\ldots,y_{M}\}\subset\mathbb{R}$. A function $f : E \rightarrow\mathbb{R}$ is called a discrete $m-$convex function if inequality $
[y_{i},\ldots,y_{i+m};f]\geq 0$
holds $\forall (m+1)$ different points $y_{i},\ldots,y_{i+m}\in E$.
\end{definition}
We extend all the aforementioned definitions up to order $(m,n)$. For that let us denote $I\times J=[a,b]\times[c,d]\subset \mathbb{R}^2$.
\begin{definition}
Let $f:I\times J\rightarrow \mathbb{R}$, be a function, then The \emph{$(m,n)-$divided difference} or \emph{divided difference of order $(m,n)$}, of a function $f$ at different points $y_i,\ldots,y_{i+m} \in I$, $z_j,\ldots,z_{j+n} \in J$ for some $i,j\in\mathbb{N}$, is ststed as
$\Delta_{(m,n)}f(y_i,z_j)=[y_i,\ldots,y_{i+m};
[z_j,\ldots,z_{j+n};f]].$
\end{definition}
\begin{definition}
A function $f:I\times J \rightarrow \mathbb{R}$, is said to be \emph{$(m,n)-$convex} or \emph{convex of order $(m,n)$},
if $\forall$ different points $y_i,\ldots,y_{i+m} \in I$ and $z_j,\ldots,z_{j+n} \in J$ we have $\Delta_{(m,n)} f(y_i,z_j) \geq 0$.

Further that the $f$ is $(m,n)-$convex if and only if $f_{(m,n)} \geq 0$, if the partial derivative $ \frac{\partial^{m+n} f}{\partial y^m \partial z^n} $ denoted by $f_{(m,n)}$ and exists.
\end{definition}
\begin{definition}
Let $E =\{y_{1},y_{2},\ldots,y_{M}\}, F =\{z_{1},z_{2},\ldots,z_{N}\}\subset\mathbb{R}$. A function $f : E\times F \rightarrow\mathbb{R}$ is called a discrete $(m,n)-$convex function if inequality $
[y_i,\ldots,y_{i+m};[z_j,\ldots,z_{j+n};f]]\geq 0 $
holds $\forall (m+1)$ different points $y_{i},\ldots,y_{i+m}\in E$ and $(n+1)$ different points $z_{j},\ldots,z_{j+n}\in F$.
\end{definition}
\begin{definition}
\emph{Finite difference of order (m,n)} of a function $f:I\times J\rightarrow \mathbb{R}$, where $h,k\in \mathbb{R}$ and $y\in I$, $z\in J$, is stated as
\begin{eqnarray*}
\Delta_{h,k}^{(m,n)}f(y,z)&=&\Delta_{h}^{(m)}(\Delta_{k}^{(n)}f(y,z))=\Delta_{k}^{(n)}(\Delta_{h}^{(m)}f(y,z))\\
&=&\sum^{m}_{i=0}\sum^{n}_{j=0}(-1)^{m+n-i-j}\binom {m}i\binom {n}j f(y+ih,z+jk).
\end{eqnarray*}
where $y+ih,z+jk \in I,J$ respectively and $i\in\{0,1,2,\ldots,m-1,m\}$; $j\in\{0,1,2,\ldots,n-1,n\}$.
Moreover, a function $f:I\times J\to \mathbb{R}$ is called the $(m,n)-$convex, if the follwoing conditions hold $\Delta^{(m,n)}_{h,k} f(y,z) \geq 0$ $\forall y\in I$, $z\in J$.
\end{definition}
\begin{definition}
\emph{Finite difference and Divided difference of $(m,n)$ order, of a sequence $(a_{ij})$ are stated as $\Delta^{(m,n)} a_{ij} = \Delta^{(m,n)}_{1,1} f(y_i,z_j)$ and $\Delta_{(m,n)} a_{ij} = \Delta_{(m,n)} f(y_i,z_j)$ respectively, where $i\in\{1,2,3,\ldots,m-1,m\}$, $j\in\{1,2,3,\ldots,n-1,n\}$. If $y_i = i$, $z_j=j$, then $f:\{1,\ldots,m\}\times\{1,\ldots,n\}\to \mathbb{R}$ is the function which is $f(i,j) = a_{ij}$. Moreover, a sequence $(a_{ij})$ is called a $(m,n)-$convex},
if following conditions hold $\Delta^{(m,n)} a_{ij} \geq 0$ for $m,n\ge0$ and $i,j\in\{1,2,3,\ldots\}$.
\end{definition}
Further, in this paper we would use the following notations: $I\times J=[a,b]\times [c,d]\subset \mathbb{R}\times \mathbb{R}$.
For some real sequence $(a_m)$, $m\in \mathbb{N}$ and $n\in \{2,3,\ldots\}$ :
$$\nabla^{(1)} a_m=\nabla a_m=a_{m}-a_{m+1},\, \quad \nabla^{(n)}a_{m}=\nabla(\nabla^{(n-1)}a_m).$$
 Also for $m$ distinct real numbers $y_i$, $i\in \{1,\ldots, m\}$ and $n\ge0$:
$$(y_k-y_i)^{\{n+1\}}=(y_k-y_i)(y_{k-1}-y_{i})\ldots(y_{k-n}-y_{i}),\quad (y_k-y_i)^{\{0\}}=1.$$
\begin{definition}
We would call the function $f:I\times J \rightarrow \mathbb{R}$ is called the $(m,n)-\nabla-$convex if inequality
$\nabla_{(m,n)} f(y_i,z_j)=(-1)^{m+n} \Delta_{(m,n)}f(y_i,z_j) \geq 0$, holds $\forall$ different points
$y_i,\ldots,y_{i+m} \in I$, $z_j,\ldots,z_{j+n} \in J$.
\end{definition}
Let us recall a useful definition from \cite[pp.~347-348]{[CANI]}.
\begin{definition}
A function $f:I\rightarrow \mathbb{R}$, is called the completely monotonic (or totally monotonic) of order $m$ or $m-$completely monotonic
if all its derivatives $f^{(i)}$ exist and satisfy
\begin{equation*}
(-1)^{i}f^{(i)}(y)\ge0, \quad y\in(0,\infty); \quad i\in\{0,1,\ldots,m\}.
\end{equation*}
\end{definition}
\begin{definition}
We wolud call the function $f:I\times J\rightarrow \mathbb{R}$ is known as completely monotonic of order $(m,n)$ or $(m,n)-$completely monotonic
if all its $f_{(i,j)}$ partial derivatives  exist and satisfy the condition below:
\begin{equation*}
(-1)^{(j+i)}f_{(i,j,)}(y,z)\ge0, \quad y,z,\in(0,\infty); \quad i\in\{0,1,2,\cdots,m-1,m \}, \quad j\in\{0,1,2,\ldots,n-1,n\}.
\end{equation*}
\end{definition}
\begin{remark}\label{nable-CMF}
It is simple to observe that the notions of completely monotonic function order $m$ and $(m,n)$ are generalized notions of $m-\nabla-$convex function and $(m,n)-\nabla-$convex function for differentiable functions respectively.
\end{remark}
\subsection{Examples} 

In present subsection, we would use variety of classes of completely monotonic functions $F=\{f_{v}:v \in
I\subset\mathbb{R}\}$ and construct examples and applications of completely monotonic functions .

\begin{example}
Let a family of functions $F_{1}= {\{\psi_{v}:\mathbb{R}\rightarrow
\mathbb{R_{+}} | v\in \mathbb{R_{+}}}\}$ which is stated as
\[\label{l2.101}\psi_{v}(y)\,=\,
\begin{array}
[c]{ll}%
\frac{e^{-vy}}{v^{m}}
\end{array}
\]
Since $(-1)^m\frac{d^{m}}{dy^{m}}\psi_{v}(y)>0$,
therefore the function $\psi_{v}$ is $m-$completely monotonic on $\mathbb{R}$, for every $v\in \mathbb{R_{+}}$.
\end{example}
\begin{example}
Let a family of functions $F_{2}=\{\phi_{v}:\mathbb{R_{+}}\rightarrow\mathbb{R} | v\in\mathbb{R_{+}}\}$ which is stated as
\begin{equation*}
  \phi_{v}(y)=
\left\{\begin{array}{rl}
\frac{v^{-y}}{(\ln v)^{m}} & ,\quad v\neq1\\
\frac{(-1)^{m}y^{m}}{m!} & ,\quad v=1
\end{array}\right.
\end{equation*}
Since $(-1)^m\frac{d^{m}}{dy^{m}}\phi_{v}(y)\geq 0$,
therefore the function $\phi_{v}$ is $m-$completely monotonic on $\mathbb{R_{+}}$ for every $v\in \mathbb{R_{+}}, y\geq 0$.
\end{example}

\begin{remark}
Other examples of completely monotonic functions include:
\begin{enumerate}
\item[(i)]$f(y)=c$ (a nonnegative real constant),\quad$\forall{y} \in \mathbb{R}$
\item[(ii)]$f(y)=\frac{\alpha}{y^{1-\alpha}}$, \quad$0\le\alpha\le1$, $y>0$
\item[(iii)]$f(y)=\frac{1}{(y+{\alpha}^2)^\beta}$, \quad$\alpha\ge0, \; \beta\ge0$, $y>0$
\item[(iv)] $f(y)=-\ln y$ \quad $\forall y \in \mathbb{R_{*}}$
\item[(v)] $f(y)=-\ln(1-1/y)$, \quad $\forall{y} \in \mathbb{R_{+}}$
\item[(vi)] $f(y)=e^{1/y}$,\quad $\forall y\in \mathbb{R_{+}}$
\end{enumerate}
\end{remark}
Let us give brief explanation the formate of our paper as follows, after introduction and some preliminary, in the second section, we consider one identity for the integral $\int P(y) f(y)dy$ which involves functions of higher order derivatives. This identity is basic tool for obtaining necessary and sufficient conditions for $(m+1)-\nabla-$convex function in which $\int P(y) f(y)dy \geq 0$ holds and only necessary condition for $(m+1)-$completely monotonic function, after second section we will work on third section in which obtain an identity for the sum $\sum_{i=1}^M\sum_{j=1}^N p_{ij}f(y_i,z_j)$ and investigate the inequality $ \sum_{i=1}^M\sum_{j=1}^N p_{ij}f(y_i,z_j) \geq 0$ for $\nabla-$convex function of order $(m,n)$ for two variables. The forth section is devoted to the integral case for $(M+1,N+1)-\nabla-$convex functions and $(M+1,N+1)-$completely monotonic functions for two variables,then consider an identity of linear functional $\Lambda(f)$ in double integral. In fifth section we state some mean value theorem of Lagrange and Cauchy types. In sixth section, consider the nonnegative functional $\Lambda(f)$ and apply this on exponentially convex functions $\psi^{(q)}$ of certain type,  and give some properties. In the last section of the paper, construct examples and applications of completely monotonic, exponentially convex functions by using various classes of functions.
\section{ \bf{Integral Case for Function of One Variable}}
In the paper
 \cite{hung1} the following result for a real sequence  $(a_M)$ was proved:
\begin{proposition}\label{proposition}
Let $p_i\in \mathbb{R}$ for $i\in\{1, \ldots ,M\}$, then the following identity for any real sequence $(a_M)$ holds:
\begin{eqnarray}
\sum_{i=1}^Mp_ia_i&=&\sum_{k=0}^{m-1} \frac{1}{k!}\nabla^{(k)}a_{M-k}
\sum_{i=1}^{M-k}(M-i)^{\{k\}}p_i  \nonumber \\ \label{2}
&+&\frac{1}{(m-1)!}\sum_{k=1}^{M-m}\left(\sum_{i=1}^k(k-i+m-1)^{\{m-1\}}p_i\right)\nabla^{(m)}a_{k}
\end{eqnarray}
 \end{proposition}
Similar result was proved in \cite{Asif-popoviciu} for the real function  involving the operator $\nabla$ and it is
a generalization of (\ref{2}) which may be stated as:
\begin{proposition}\label{lemma}
Let $m, M$ be  integers such that $m\leq M$ and let $p_i$ be real numbers for $i\in\{1,2,\cdots,M\}$ . Let $f$ be a function and $y_i$ be non mutual elements from interval $I$ for $i\in\{1,2,\ldots,M\}$, then following identity holds:
\begin{eqnarray}\label{four}
\sum_{i=1}^M p_i f(y_i)&=&\sum_{k=0}^{m-1}\left(\sum_{j=1}^{M-k}p_j(y_M-y_j)^{\{k\}}\right)
\nabla_{(k)}f(y_{M-k})  \\ \nonumber
 &+&\sum_{k=1}^{M-m}\left(\sum_{j=1}^kp_j(y_{k+m-1}-y_{j})^{\{m-1\}}\right)\nabla_{(m)}f(y_k)
(y_{k+m}-y_{k})
\end{eqnarray}
\end{proposition}
We can also prove an integral identity that is analogous to the above formula.
\begin{theorem}\label{thm1}
Let a function $f \in C^{(m+1)}$ and $P,f:I \to \mathbb{R}$, both be integrable functions, then
\begin{eqnarray}\nonumber
\int_a^b f(y)P(y) dy&=&\sum_{i=0}^{m}\left(\int_a^b P(y) \frac{(b-y)^i}{i!}dy\right) (-1)^i f^{(i)}(b)\\ \label{var-1}
&+& \int_a^b \left( \int_a^s P(y)\frac{(s-y)^m}{m!}dy\right)(-1)^{m+1}f^{(m+1)}(s) ds
\end{eqnarray}
\end{theorem}
\begin{proof}
The function $f$ by using  Taylor expansion can be represented as
\begin{eqnarray*}
f(y)&=&\sum_{i=0}^{m} f^{(i)}(b) \frac{(y-b)^i}{i!}
+ \int_b^y f^{(m+1)}(s) \frac{(y-s)^m}{m!} ds \\
&=& \sum_{i=0}^{m} (-1)^i f^{(i)}(b) \frac{(b-y)^i}{i!}
+ \int_y^b (-1)^{m+1}f^{(m+1)}(s) \frac{(s-y)^m}{m!} ds.
\end{eqnarray*}
Multiplying the above equation by $P$ and integrate it over $[a,b]$, then
\begin{eqnarray*}
\int_a^b f(y)P(y) dy&=&
\sum_{i=0}^{m} (-1)^i f^{(i)}(b)\int_a^b P(y) \frac{(b-y)^i}{i!}dy
\\&+& \int_a^b \left(\int_y^b (-1)^{m+1}f^{(m+1)}(s)ds\right) P(y)\frac{(s-y)^m}{m!} dy \\
&=&\sum_{i=0}^{m}\left(\int_a^b P(y) \frac{(b-y)^i}{i!}dy\right) (-1)^i f^{(i)}(b)\\
&+&  \int_a^b\left( \int_a^s P(y)\frac{(s-y)^m}{m!}dy\right)(-1)^{m+1}f^{(m+1)}(s) ds.
\end{eqnarray*}
We used the Fubini theorem in above last equation for the variables $y$ and $s$.
\end{proof}
The following theorem is generalized form of result (see \cite[pp. 121--122]{pec}) and (see also \cite{Asif-popoviciu}).
\begin{theorem} \label{thm2}
Let supposition of the Theorem $\ref{thm1}$ be true, then following inequality holds
\begin{equation}\label{positive-1}
\int_a^b f(y)P(y) dy \geq 0
\end{equation}
 for all $(m+1)-\nabla-$convex function $f$, if and only if
\begin{eqnarray}\label{7}
\int_a^b P(y) \frac{(b-y)^i}{i!}dy=0,&\quad i\in \{1,\ldots,m\}\\
\label{8}
\int_a^s P(y)\frac{(s-y)^m}{m!}dy \ge0,&\quad \forall s\in[a,b].
\end{eqnarray}
\end{theorem}
\begin{proof}
If (\ref{7}) holds, then the first sum is zero in (\ref{var-1}) and required inequality (\ref{positive-1}) holds by applying (\ref{8}).

Conversely, if substitute the following functions in (\ref{positive-1}). Then
\begin{eqnarray*}
f_{1}(y)=\frac{(b-y)^{i} }{i!}\quad \text{and}\quad f_{2}=-f_{1}
 \end{eqnarray*}
for $0\leq i\leq m$ such that $(-1)^{m+1}f_{l}^{(m+1)}(y)\geq0,\quad l\in \{1,2\},$  then obtain required equality (\ref{7}) i.e.
$$\int _{a}^{b}P(y)\frac{(b-y)^{i} }{i!} dy=0,\quad 0\leq i\leq m.$$
We get the last inequality (\ref{8}) by assuming below function in (\ref{positive-1}), where $s\in[a,b]$
 $$
f_3(y)= \left\{ \begin{array}{rl}
\frac{(s-y)^{m} }{m!}, & y<s\\
0, & y\ge s.
\end{array} \right.
$$
\end{proof}
\begin{theorem} \label{thm2}
Let supposition of the Theorem $\ref{thm1}$ be true, then following inequality holds;
\begin{equation}\label{positive-11}
\int_a^b P(y)f(y) dy \geq 0
\end{equation}
 for every completely monotonic function $f$ of order $m+1$ if
\begin{eqnarray}\label{71}
\int_a^b P(y) \frac{(b-y)^i}{i!}dy=0,&\quad i\in \{1,\ldots,m\}\\
\label{81}
\int_a^s P(y)\frac{(s-y)^m}{m!}dy \ge0,&\quad \forall s\in[a,b].
\end{eqnarray}
\end{theorem}
\begin{proof}
If (\ref{71}) holds, then the first sum is zero in (\ref{var-1}) and the required inequality (\ref{positive-11}) holds by using (\ref{81}).
\end{proof}

\begin{remark}
 It is a known fact that every $\nabla-$convex function is a subclass of completely monotonic function and this is a subclass of exponentially convex function and hence the previous result also holds for exponentially convex functions \cite{Method}. Moreover, every completely monotonic function is log-convex so the stated result also holds for every log-convex function \cite{[CMF]}.
\end{remark}

\section{\bf{Discrete Case for Functions of Two Variables}}
Under the given topic, we would consider a two variables discrete function that are defined on $I_{1}\times J_{1}\subset\mathbb{R}\times\mathbb{R}$. Firstly, we will get an identity for $\sum_{i=1}^M\sum_{j=1}^N p_{ij}f(y_i,z_j)$ in which involves divided differences and consider necessary and sufficient conditions of next theorem for inequality of Popoviciu type characterisation of positivity of sums for higher order discrete functions $\sum_{i=1}^M\sum_{j=1}^N p_{ij}f(y_i,z_j) \geq 0$ holds, for every $(m,n)-\nabla-$convex function.
\begin{theorem}\label{discrete-id}
Let $p_{ij}\in \mathbb{R}$ and $f:I_{1}\times J_{1}\to \mathbb{R}$ be discrete function, where $i\in\{1,2,3,\ldots,M-1,M\}$ and $ j\in\{1,2,3,\cdots,N-1,N\}$, then following identity holds:\begin{align} \label{double1}
&\sum_{i=1}^M\sum_{j=1}^N p_{ij}f(y_i,z_j)\\ \nonumber
&=\sum_{k=0}^{n-1} \sum_{t=0}^{m-1}   \sum_{s=1}^{M-t} \sum_{r=1}^{N-k} p_{sr} (z_N-z_r)^{\{k\}}
(y_M-y_{s})^{\{t\}}   \nabla_{(t,k)} f(y_{M-t},z_{N-k})\\ \nonumber
&+
\sum_{k=0}^{n-1} \sum_{t=1}^{M-m} \sum_{s=1}^{t} \sum_{r=1}^{N-k} p_{sr} (z_N-z_r)^{\{k\}} (y_{t+m-1}-y_s)^{\{m-1\}}   \nabla_{(m,k)} f(y_{t}, z_{N-k})(y_{t+m}-y_{t})\\ \nonumber
&+ \sum_{k=1}^{N-n} \sum_{t=0}^{m-1}\sum_{s=1}^{M-t} \sum_{r=1 }^k p_{sr} (z_{k+n-1}-z_r)^{\{n-1\}} (y_M-y_{s})^{\{t\}} \nabla_{(t,n)} f(y_{M-t}, z_{k}) (z_{k+n}-z_k)\\ \nonumber
&+\sum_{k=1}^{N-n} \sum_{t=1}^{M-m} \sum_{s=1}^{t} \sum_{r=1 }^k p_{sr} (z_{k+n-1}-z_r)^{\{n-1\}} (y_{t+m-1}-y_s)^{\{m-1\}}\nabla_{(m,n)} f(y_{t}, z_{k}) \times \\ \nonumber
& \times(y_{t+m}-y_{t})(z_{k+n}-z_k).
 \end{align}
where  $(y_i,z_j)\in I_{1}\times J_{1}$  are distinct points.
\end{theorem}
\begin{proof}
 We have
\[ \sum_{i=1}^M\sum_{j=1}^N p_{ij}f(y_i,z_j)=\sum_{i=1}^M\left( \sum_{j=1}^N q_j G_i(z_j) \right), \]
where $p_{ij}=q_j$ and $G_i : z \mapsto f(y_i,z)$. Using (\ref{four}) in the inner sum we get
\begin{align*}
& \sum_{i=1}^M\sum_{j=1}^N p_{ij}f(y_i,z_j)=\sum_{i=1}^M \sum_{k=0}^{n-1} \left( \sum_{j=1}^{N-k} q_j (z_N-z_j)^{\{k\}}\right) \nabla_{(k)} G_i(z_{N-k})\\
& +
\sum_{i=1}^M \sum_{k=1}^{N-n} \left(\sum_{j=1}^k q_j (z_{k+n-1}-z_j)^{\{n-1\}}\right) \nabla_{(n)} G_i(z_{k})(z_{k+n}-z_k)
 \\& =
\sum_{k=0}^{n-1} \left( \sum_{i=1}^M  \left( \sum_{j=1}^{N-k} q_j (z_N-z_j)^{\{k\}}\right) \nabla_{(k)} G_i(z_{N-k})\right)\\
&+ \sum_{k=1}^{N-n} \left( \sum_{i=1}^M  \left( \sum_{j=1 }^k q_j (z_{k+n-1}-z_j)^{\{n-1\}}  \right) \nabla_{(n)} G_i(z_{k})(z_{k+n}-z_k) \right)
  \\
&=\sum_{k=0}^{n-1} \left( \sum_{i=1}^M w_i F(y_i) \right) + \sum_{k=1}^{N-n} \left( \sum_{i=1}^M v_i H(y_i) \right)
 \end{align*}
where $w_i= \sum_{j=1}^{N-k} q_j (z_N-z_j)^{\{k\}}= \sum_{j=1}^{N-k} p_{ij} (z_N-z_j)^{\{k\}}$,\quad$v_i=\sum_{j=1 }^k q_j (z_{k+n-1}-z_j)^{\{n-1\}}$,\\
$F(y_i)=  \nabla_{(k)} G_i(z_{N-k})$, and $H(y_i)= \nabla_{(n)} G_i(z_{k})(z_{k+n}-z_k)$.

Using again (\ref{four}) on  inner sums, then we have
\begin{align*}
&\sum_{i=1}^M\sum_{j=1}^N p_{ij}f(y_i,z_j)\\
&= \sum_{k=0}^{n-1} \sum_{r=0}^{m-1} \left( \sum_{i=1}^{M-r} w_i(y_M-y_{i})^{\{r\}} \right) \nabla_{(r)} F(y_{M-r})\\
& +
\sum_{k=0}^{n-1} \sum_{r=1}^{M-m} \left( \sum_{i=1}^{r} w_i(y_{r+m-1}-y_i)^{\{m-1\}} \right) \nabla_{(m)}F(y_{r})(y_{r+m}-y_{r})\\
&+  \sum_{k=1}^{N-n} \sum_{t=0}^{m-1} \left(\sum_{i=1}^{M-t} v_i (y_M-y_{i})^{\{t\}} \right) \nabla_{(t)} H(y_{M-t})
\\
&+\sum_{k=1}^{N-n} \sum_{t=1}^{M-m} \left( \sum_{i=1}^{t} v_i (y_{t+m-1}-y_i)^{\{m-1\}} \right) \nabla_{(m)} H(y_{t})(y_{t+m}-y_{t}) \end{align*}
\begin{align*}
&=\sum_{k=0}^{n-1} \sum_{r=0}^{m-1}   \sum_{i=1}^{M-r} \sum_{j=1}^{N-k} p_{ij} (z_N-z_j)^{\{k\}}
(y_M-y_{i})^{\{r\}}   \nabla_{(r,k)} f(y_{M-r},z_{N-k}) \\
&+
\sum_{k=0}^{n-1} \sum_{r=1}^{M-m}  \sum_{i=1}^{r} \sum_{j=1}^{N-k} p_{ij} (z_N-z_j)^{\{k\}} (y_{r+m-1}-y_i)^{\{m-1\}}   \nabla_{(m,k)} f(y_{r}, z_{N-k})(y_{r+m}-y_{r})\\
&+ \sum_{k=1}^{N-n} \sum_{t=0}^{m-1} \sum_{i=1}^{M-t} \sum_{j=1 }^k p_{ij} (z_{k+n-1}-z_j)^{\{n-1\}} (y_M-y_{i})^{\{t\}} \nabla_{(t,n)} f(y_{M-t}, z_{k}) (z_{k+n}-z_k)
\\
&+\sum_{k=1}^{N-n} \sum_{t=1}^{M-m} \sum_{i=1}^{t} \sum_{j=1 }^k p_{ij} (z_{k+n-1}-z_j)^{\{n-1\}} (y_{t+m-1}-y_i)^{\{m-1\}}\nabla_{(m,n)} f(y_{t}, z_{k}) \times \\
& \times(y_{t+m}-y_{t})(z_{k+n}-z_k).
 \end{align*}
If change $i \rightarrow s$, $j\rightarrow r$ in all sums and put $r \rightarrow t$ in first and second sums, then obtain the required identity(\ref{double1}).
\end{proof}
\begin{remark}\label{rk1}
If we simply put $f(y_i,z_j)=f(y_i)g(z_i)$ in Theorem \ref{discrete-id}, then we obtain similar result for two $f$ and $g$ functions as follows.
\end{remark}
\begin{corollary}\label{cor1}
Let $p_{ij}\in \mathbb{R}$, $f:I_{1} \to \mathbb{R}$ and $g:J_{1} \to \mathbb{R}$ be two discrete functions, where $i\in\{1,2,3,\ldots,M-1,M\}$ and $ j\in\{1,2,3,\cdots,N-1,N\}$, then following identity holds:
\begin{align*}\nonumber
&\sum_{i=1}^M\sum_{j=1}^N p_{ij}f(y_i)g(z_j)\\ \nonumber
&=\sum_{k=0}^{n-1} \sum_{t=0}^{m-1}   \sum_{s=1}^{M-t} \sum_{r=1}^{N-k} p_{sr} (y_M-y_{s})^{\{t\}}   \nabla_{(t)} f(y_{M-t}) (z_N-z_r)^{\{k\}}\nabla_{(k)} g(z_{N-k}) \\ \nonumber
&+
\sum_{k=0}^{n-1} \sum_{t=1}^{M-m} \sum_{s=1}^{t} \sum_{r=1}^{N-k} p_{sr} (z_N-z_r)^{\{k\}}\nabla_{(k)} g( z_{N-k}) (y_{t+m-1}-y_s)^{\{m-1\}}   \nabla_{(m)} f(y_{t})(y_{t+m}-y_{t})\\ \nonumber
&+ \sum_{k=1}^{N-n} \sum_{t=0}^{m-1}\sum_{s=1}^{M-t} \sum_{r=1 }^k p_{sr} (z_{k+n-1}-z_r)^{\{n-1\}} \nabla_{(n)} g(z_{k}) (z_{k+n}-z_k)(y_M-y_s)^{\{t\}} \nabla_{(t)} f(y_{M-t})
\\ \nonumber
&+\sum_{k=1}^{N-n} \sum_{t=1}^{M-m} \sum_{s=1}^{t} \sum_{r=1 }^k p_{sr} (z_{k+n-1}-z_r)^{\{n-1\}}\nabla_{(n)} g( z_{k})(z_{k+n}-z_k) \times \\ \nonumber & \times(y_{t+m-1}-y_s)^{\{m-1\}}\nabla_{(m)} f(y_{t})
(y_{t+m}-y_{t}).
\end{align*}
where $(y_i,z_j)\in I_{1}\times J_{1}$ are distinct points.
\end{corollary}
\begin{remark}\label{rk2}
We can easily obtain similar result for sequences by simply considering the function $a_{ij}=f(i,j)$ in previous theorem. Moreover, we can also split this sequence into two sequences as special case by using $a_{ij}=a_ib_j$ instead of $f(y_i,z_j)=f(y_i)g(z_j)$ is previous corollary.
\end{remark}
\begin{theorem}\label{thm3}
Let $p_{ij}\in \mathbb{R}$ and $f:I_{1}\times J_{1}\to \mathbb{R}$ be discrete function, where $i\in \{1,2,3,\ldots,M-1,M\}$, $j\in \{1,2,3,\cdots,N-1 ,N\}$ and $I_{1}=\{y_{M-r},y_{M-r+1},\ldots,y_{M}\}$, $J_{1}=\{z_{N-k},z_{N-k+1},\ldots,z_{N}\}$ and $y_{M-r} < \cdots <y_M$, $z_{N-k}<\cdots <z_N$, then the below inequality holds for all $(m,n)-\nabla-$convex function $f$.
\begin{equation}\label{9}
 \sum_{i=1}^M\sum_{j=1}^N p_{ij}f(y_i,z_j) \geq 0
 \end{equation}
if and only if
\begin{eqnarray}\label{10}
\sum_{s=1}^{M-t} \sum_{r=1}^{N-k} p_{sr} (z_N-z_r)^{\{k\}}(y_M-y_{s})^{\{t\}} &=&0,
\;\;\;\begin{array}{ll}
k\in \{0,1,2,\cdots ,n-1\}\\ t\in \{0,1,2,\cdots ,m-1\}
\end{array}\\ \label{11}
\sum_{s=1}^{t} \sum_{r=1 }^{N-k} p_{sr} (z_N-z_r)^{\{k\}} (y_{t+m-1}-y_s)^{\{m-1\}} &=&0,
\;\;\;\begin{array}{ll}
k\in\{0,1,2,\cdots ,n-1\}\\ t\in\{1,2,3,\cdots ,M-m\}
\end{array}\\ \label{12}
\sum_{s=1}^{M-t} \sum_{r=1 }^k p_{sr} (z_{k+n-1}-z_r)^{\{n-1\}} (y_M-y_s)^{\{t\}} &=&0,
\;\;\;\begin{array}{ll}
k\in\{1,2,3,\cdots ,N-n\}\\ t\in\{0,1,2,\cdots ,m-1\}
\end{array}\\ \label{13}
\sum_{s=1}^t \sum_{r=1}^k p_{sr}(y_{t+m-1}-y_{s})^{\{m-1\}}(z_{k+n-1}-z_{r})^{\{n-1\}} &\geq&0,
\;\;\;\begin{array}{ll}
k\in\{1,2,3,\cdots ,N-n\}\\ t\in\{1,2,3,\cdots ,M-m\}.
\end{array}
\end{eqnarray}
\end{theorem}
\begin{proof}
If (\ref{10}), (\ref{11}) and (\ref{12}) hold, then first, second and third sums are zero in (\ref{double1}), then by using (\ref{13}) we obtain the required inequality (\ref{9}).

Conversely, if substitute the following functions in (\ref{9}). Then we obtain the required equality (\ref{10})
\begin{eqnarray*}
f_{1}(y_{s},z_{r})=(z_N-z_r)^{\{k\}}(y_M-y_{s})^{\{t\}}\quad \text{and} \quad f_{2}=-f_{1}
 \end{eqnarray*}
 for $0\leq t\leq m-1$ and $0\leq k\leq n-1$ such that $\nabla_{(m,n)} f_{l}\geq 0,\quad l\in \{1,2\}$    $$\sum_{s=1}^{M-t} \sum_{r=1}^{N-k} p_{sr} (z_N-z_r)^{\{k\}}(y_M-y_{s})^{\{t\}}=0,\quad 0\leq t\leq m-1,\;\;\;0\leq k\leq n-1.$$
 In the similar manner, if take the following functions in (\ref{9})  for $0\leq k\leq n-1$ and $1\leq t\leq M-m$
  $$
f_{3}(y_{s},z_{r})= \left\{ \begin{array}{rl}
(z_N-z_r)^{\{k\}} (y_{t+m-1}-y_s)^{\{m-1\}}, & s<t \\
0, & s\geq t
\end{array} \right.
$$
 $$
f_{4}=-f_{3}
$$
such that $\nabla_{(m,n)} f_{l}\geq 0,\quad l\in \{3,4\}$, we obtain the equality (\ref{11}) i.e. $$\sum_{s=1}^{t} \sum_{r=1 }^k p_{sr} (z_N-z_r)^{\{k\}} (y_{t+m-1}-y_s)^{\{m-1\}}=0,\quad 0\leq k\leq n-1,\quad 1\le t \le M-m.$$
Similarly, if take the following functions in (\ref{9}) for $1\leq k\leq N-n $ and $0\le t\le m-1$
  $$
f_{5}(y_{s},z_{r})= \left\{ \begin{array}{rl}
(z_{k+n-1}-z_r)^{\{n-1\}} (y_M-y_s)^{\{t\}}, & r<k \\
0, & r\geq k
\end{array} \right.
$$
 $$
f_{6}=-f_{5}
$$
such that $\nabla_{(m,n)} f_l\geq 0,\quad l\in \{5,6\}$, we obtain the equality (\ref{12}) as above, i.e. $$\sum_{s=1}^{M-t} \sum_{r=1 }^k p_{sr} (z_{k+n-1}-z_r)^{\{n-1\}} (y_M-y_s)^{\{t\}}=0,\;\;\;1\leq k\leq N-n,\quad 0\le t\le m-1.$$
We get the last inequality (\ref{13}) by considering the following function in (\ref{9}) for $1\le k\le N-n$ and $1\le t\le M-m$
 $$
f_7(y_{s},z_{r})= \left\{ \begin{array}{rl}
(y_{t+m-1}-y_{s})^{\{m-1\}}(z_{k+n-1}-z_{r})^{\{n-1\}}, & s<t,\quad\;\; \quad r<k \\
0, & s\geq t\quad \text{or}\quad r\geq k.
\end{array} \right.
$$
\end{proof}
\begin{remark}
Similar remarks as given in Remarks \ref{rk1} and \ref{rk2} also hold for this result. Hence we can state a Corollary similar to Corollary \ref{cor1} for previous main theorem as well.
\end{remark}

\section{\bf{Integral Case for Functions of Two Variables}}
Under continuing topic, we can suppose function in $x$ and $y$ variables which is defined on the interval $I\times J=[a,b]\times[c,d]$. Moreover, $m,n,M,N\in \mathbb{N}\cup \{ 0\}$ throughout the section, and usable notations are: $$ f_{(0,0)}=f \quad f_{(1,0)}=\frac{\partial f}{\partial y}, \quad f_{(0,1)}=\frac{\partial f}{\partial z}, \quad f_{(1,1)}=\frac{\partial^2 f}{\partial y \partial z}=\frac{\partial^2 f}{\partial z \partial y},\quad f_{(i,j)}=\frac{\partial^{i+j} f}{\partial y^{i} \partial z^{j}}=\frac{\partial^{i+j} f}{\partial z^{j} \partial y^{i}}$$
Now we recall a result from \cite{Asif-popoviciu} which would be helpful to prove our next main result:
\begin{lemma}\label{thm4.2}
Let $f$ has continuous partial derivatives $f_{(i,j)}$ and $P,f:I\times J\rightarrow\mathbb{R}$ be both integrable functions, where $i\in\{0,1,2,\cdots,M,M+1\}$ and $j\in\{0,1,2,\cdots,N,N+1\}$, then
 \begin{eqnarray}
\int_a^b \int_c^d P(y,z)f(y,z) dz dy &=& \sum_{i=0}^{M} \sum_{j=0}^{N}  \int_a^b \int_c^d P(s,t)  \frac{(s-a)^i}{i!}   \frac{(t-c)^j}{j!} f_{(i,j)}(a,c) dt ds \\ \nonumber
&+&\sum_{j=0}^{N}\int_a^b  \int_y^b  \int_c^d P(s,t)   \frac{(s-y)^M}{M!}\frac{(t-c)^j}{j!} f_{(M+1,j)}(y,c) dt  ds  dy \\ \nonumber
&+&\sum_{i=0}^{M} \int_c^d  \int_a^b  \int_z^d P(s,t)   \frac{(s-a)^i}{i!} \frac{(t-z)^N}{N!} f_{(i,N+1)}(a,z) dt ds  dz\\ \nonumber
&+&\int_a^b  \int_c^d \int_y^b    \int_z^d P(s,t)  \frac{(s-y)^M}{M!} \frac{(t-z)^N}{N!} f_{(M+1,N+1)}(y,z) dt  ds dz  dy.
\end{eqnarray}
\end{lemma}
\begin{theorem}\label{th1}
Let $f$ has continuous partial derivatives $f_{(i,j)}$ and $P,f:I\times J\rightarrow\mathbb{R}$ be both integrable functions, where $i\in\{0,1,2,\cdots,M,M+1\}$ and $j\in\{0,1,2,\ldots,N,N+1\}$, then
 \begin{eqnarray}\label{dodu}
 &&\int_a^b \int_c^d P(y,z)f(y,z) dz\, dy\\ \nonumber
&&=
\sum_{i=0}^{M} \sum_{j=0}^{N}  \int_a^b \int_c^d P(y,z) \frac{(b-y)^i}{i!}   \frac{(d-z)^j}{j!}  (-1)^{i+j} f_{(i,j)}(b,d) dz \, dy \\ \nonumber
&&+
 \sum_{j=0}^{N}\int_a^b  \int_a^s  \int_c^d P(y,z) \frac{(s-y)^M}{M!}\frac{(d-z)^j}{j!} (-1)^{M+j+1} f_{(M+1,j)}(s,d)  dz \, dy \, ds \\ \nonumber
&&+
\sum_{i=0}^{M} \int_c^d  \int_a^b  \int_c^t P(y,z) \frac{(b-y)^i}{i!} \frac{(t-z)^N}{N!} (-1)^{i+N+1} f_{(i,N+1)}(b,t)  dz\, dy\,  dt\\ \nonumber
&&+
\int_a^b  \int_c^d \int_a^s    \int_c^t P(y,z) \frac{(s-y)^M}{M!} \frac{(t-z)^N}{N!} (-1)^{M+N}f_{(M+1,N+1)}(s,t)  dz \, dy\, dt\,  ds. \\ \nonumber
\end{eqnarray}
\end{theorem}

\begin{proof}
We restate the identity given in Lemma \ref{thm4.2} as follows
 \begin{eqnarray}
 &&\int_A^B \int_C^D P(y,z)f(y,z) dz\, dy\\ \nonumber
&&=
\sum_{i=0}^{M} \sum_{j=0}^{N}  \int_A^B \int_C^D P(s,t) \frac{(s-A)^i}{i!}   \frac{(t-C)^j}{j!} f_{(i,j)}(A,C)  dt \, ds \\ \nonumber
&&+
 \sum_{j=0}^{N}\int_A^B  \int_y^B  \int_C^D P(s,t) \frac{(s-y)^M}{M!}\frac{(t-C)^j}{j!} f_{(M+1,j)}(y,C)  dt \, ds \, dy \\ \nonumber
&&+
\sum_{i=0}^{M} \int_C^D  \int_A^B  \int_z^D P(s,t) \frac{(s-A)^i}{i!} \frac{(t-z)^N}{N!} f_{(i,N+1)}(A,z)  dt\, ds\,  dz\\ \nonumber
&&+
\int_A^B  \int_C^D \int_y^B    \int_z^D P(s,t) \frac{(s-y)^M}{M!} \frac{(t-z)^N}{N!} f_{(M+1,N+1)}(y,z)  dt \, ds\, dz\,  dy. \\ \nonumber
\end{eqnarray}
Let us substitute $[A,B]=[b,a]$ and $[C,D]=[d,c]$. Then $\int_A^B = \int_b^a= -\int_a^b$ etc. and we change the variables names $y\leftrightarrow s$, $z\leftrightarrow t$, then the right hand side:
\begin{eqnarray}\label{mainth4}
 &&\int_b^a \int_d^c P(y,z)f(y,z) dz\, dy\\ \nonumber
&&=
\sum_{i=0}^{M} \sum_{j=0}^{N}  \int_b^a \int_d^c P(y,z) \frac{(y-b)^i}{i!}   \frac{(z-d)^j}{j!} f_{(i,j)}(b,d)  dz \, dy \\ \nonumber
&&+
 \sum_{j=0}^{N}\int_b^a  \int_s^a  \int_d^c P(y,z) \frac{(y-s)^M}{M!}\frac{(z-d)^j}{j!} f_{(M+1,j)}(s,d)  dz \, dy \, ds \\ \nonumber
&&+
\sum_{i=0}^{M} \int_d^c  \int_b^a  \int_t^c P(y,z) \frac{(y-b)^i}{i!} \frac{(z-t)^N}{N!} f_{(i,N+1)}(b,t)  dz\, dy\,  dt\\ \nonumber
&&+
\int_b^a  \int_d^c \int_s^a    \int_t^c P(y,z) \frac{(y-s)^M}{M!} \frac{(z-t)^N}{N!} f_{(M+1,N+1)}(s,t)  dz \, dy\, dt\,  ds. \\ \nonumber
\end{eqnarray}
The left hand side of the \ref{mainth4} may be written as
\begin{eqnarray*}
\int_b^a \int_d^c f(y,z)P(y,z) dz\, dy=\int_a^b \int_c^d (-1)^2f(y,z)P(y,z) dz\, dy= \int_a^b \int_c^d f(y,z)P(y,z) dz\, dy\\
 \end{eqnarray*}
We can write the first summand on right hand side as
\begin{eqnarray*}
&&\sum_{i=0}^{M} \sum_{j=0}^{N}  \int_b^a \int_d^c P(y,z) \frac{(y-b)^i}{i!}   \frac{(z-d)^j}{j!} f_{(i,j)}(b,d)  dz \, dy \\
&&= \sum_{i=0}^{M} \sum_{j=0}^{N}  \int_a^b \int_c^d (-1)^2 P(y,z) (-1)^i \frac{(b-y)^i}{i!}  (-1)^j  \frac{(d-z)^j}{j!} f_{(i,j)}(b,d) dz \, dy \\
&&= \sum_{i=0}^{M} \sum_{j=0}^{N}  \int_a^b \int_c^d (-1)^{i+j} P(y,z) \frac{(b-y)^i}{i!}     \frac{(d-z)^j}{j!}  f_{(i,j)}(b,d)   dz \, dy \\
 \end{eqnarray*}
Also write the second summand on right hand side as

\begin{eqnarray*}
&&\sum_{j=0}^{N}\int_b^a  \int_s^a  \int_d^c P(y,z) \frac{(y-s)^M}{M!}\frac{(z-d)^j}{j!} f_{(M+1,j)}(s,d)  dz \, dy \, ds \\
&&=\sum_{j=0}^{N}\int_a^b  \int_a^s  \int_c^d (-1)^3 P(y,z) (-1)^M \frac{(s-y)^M}{M!} (-1)^j\frac{(d-z)^j}{j!}  f_{(M+1,j)}(s,d) dz \, dy \, ds \\
&&=\sum_{j=0}^{N}\int_a^b  \int_a^s  \int_c^d (-1)^{M+1+j} P(y,z) \frac{(s-y)^M}{M!}  \frac{(d-z)^j}{j!} f_{(M+1,j)}(s,d)    dz \, dy \, ds
 \end{eqnarray*}
Similarly the third summand is rewritten as
\begin{eqnarray*}
&&\sum_{i=0}^{M} \int_d^c  \int_b^a  \int_t^c P(y,z) \frac{(y-b)^i}{i!} \frac{(z-t)^N}{N!} f_{(i,N+1)}(b,t)  dz\, dy\,  dt\\
&&=
\sum_{i=0}^{M} \int_c^d  \int_a^b  \int_c^t (-1)^3 P(y,z)   (-1)^i \frac{(b-y)^i}{i!} (-1)^N \frac{(t-z)^N}{N!} f_{(i,N+1)}(b,t) dz\, dy\,  dt\\
 &&=
\sum_{i=0}^{M} \int_c^d  \int_a^b  \int_c^t (-1)^{N+1+i} P(y,z)  \frac{(b-y)^i}{i!} \frac{(t-z)^N}{N!} f_{(i,N+1)}(b,t)    dz\, dy\,  dt
 \end{eqnarray*}
Finally, last summand on right side rewritten as

\begin{eqnarray*}
&&\int_b^a  \int_d^c \int_s^a    \int_t^c P(y,z) \frac{(y-s)^M}{M!} \frac{(z-t)^N}{N!} f_{(M+1,N+1)}(s,t)  dz \, dy\, dt\,  ds \\
&&= \int_a^b  \int_c^d \int_a^s    \int_c^t (-1)^4 P(y,z) (-1)^M\frac{(s-y)^M}{M!} (-1)^N\frac{(t-z)^N}{N!} f_{(M+1,N+1)}(s,t)  dz \, dy\, dt\,  ds\\
&&= \int_a^b  \int_c^d \int_a^s    \int_c^t (-1)^{M+N} P(y,z) \frac{(s-y)^M}{M!} \frac{(t-z)^N}{N!} f_{(M+1,N+1)}(s,t)  dz \, dy\, dt\,  ds\\
\end{eqnarray*}
By substituting all these expression in \ref{mainth4} we would arrive at our required result.
\end{proof}
\begin{remark}
\enumerate
\item This result can also be obtained by using Taylor series expansion and by using Mathematical Induction.
\item If in Theorem \ref{th1} we replace $f(y,z)$ by $f(y)g(z)$, then we obtain the following statement.
\end{remark}

\begin{corollary}
Let $g\in C^{(N+1)}(J)$, $f\in C^{(M+1)}(I)$, be two different functions and $P:I\times J \to \mathbb{R}$, be an integrable function, then state following identity as:
\begin{eqnarray*}\nonumber
 &&\int_a^b \int_c^d f(y,z)P(y,z) dz\, dy\\ \nonumber
&&=
\sum_{i=0}^{M} \sum_{j=0}^{N}  \int_a^b \int_c^d P(y,z) (-1)^{i+j} \frac{(d-z)^j}{j!} g^{(j)}(d) \frac{(b-y)^i}{i!} f^{(i)}(b)  dz \, dy \\ \nonumber
&&+
 \sum_{j=0}^{N}\int_a^b  \int_a^s  \int_c^d P(y,z) (-1)^{M+1+j} \frac{(d-z)^j}{j!}g^{(j)}(d) \frac{(s-y)^M}{M!}f^{(M+1)}(y) dz \, dy \, ds \\ \nonumber
&&+
\sum_{i=0}^{M} \int_c^d  \int_a^b  \int_c^t P(y,z) (-1)^{N+1+i} \frac{(t-z)^N}{N!}g^{(N+1)}(z) \frac{(b-y)^i}{i!}f^{(i)}(b)  dz\, dy\,  dt\\
&&+
\int_a^b  \int_c^d \int_a^s    \int_c^t P(y,z) (-1)^{M+N} \frac{(t-z)^N}{N!}g^{(N+1)}(z) \frac{(s-y)^M}{M!}f^{(M+1)}(y)  dz \, dy\, dt\,  ds.
\end{eqnarray*}
\end{corollary}

We obtain necessary and sufficient conditions by using results of previous theorem that $\Lambda (f) \geq 0$ holds  $\forall (M+1,N+1)-\nabla-$convex function and only necessary condition $\forall (M+1,N+1)-$completely monotonic function for two-variables function.
\begin{theorem}\label{positive}
Let suppositions of Theorem $\ref{th1}$ be true, then following inequality holds;
\begin{align} \label{3-3}
\Lambda(f)=\int_{a}^{b} \int_c^d P(y,z) f(y,z)dz\, dy\geq 0
\end{align}
 for all $(M+1,N+1)-\nabla-$convex function $f$ on $I\times J$, if and only if
\begin{eqnarray} \label{3-4}
\int_a^b \int_c^d P(y,z) \frac{(b-y)^i}{i!}   \frac{(d-z)^j}{j!} dz \, dy &=&0, \,\,\, \,\,i\,\,\in \{0,\ldots,M\};\,\,\, j\in \{0,\ldots,N\} \\ \label{3-5}
\int_a^s \int_c^d P(y,z) \frac{(s-y)^M}{M!}\frac{(d-z)^j}{j!} dz \, dy&=&0,\,\,\,\,\,\, j\in \{0,\ldots,N\};\,\,\,\forall \,\,s\in[a,b] \\ \label{3-6}
\int_a^b \int_c^t P(y,z) \frac{(b-y)^i}{i!} \frac{(t-z)^N}{N!} dz \, dy&=&0,\,\,\, \,\,i\,\,\in \{0,\ldots,M\};\,\,\,\forall \,\,t\in[c,d] \\ \label{3-7}
\int_a^s \int_c^t P(y,z) \frac{(s-y)^M}{M!} \frac{(t-z)^N}{N!} dz \, dy&\geq& 0,\,\,\,\forall \,\,s\in[a,b];\,\,\,\forall \,\,t\in[c,d].
\end{eqnarray}
\end{theorem}
\begin{proof}
If (\ref{3-4}), (\ref{3-5}) and (\ref{3-6}) hold, then first, second and third sums are zero in (\ref{dodu}), then by using (\ref{3-7}) we obtain the required inequality (\ref{3-3}).

Conversely, if substitute the following functions in (\ref{3-3}). Then
\begin{eqnarray*}
f^{1}(y,z)&=&\frac{(b-y)^{m} }{m!} \frac{(d-z)^{n} }{n!}\quad \text{and}\quad f^{2}=-f^{1}
 \end{eqnarray*}
 for $0\leq n\leq N$ and $0\leq m\leq M$ such that $(-1)^{M+N}f^{l}_{(M+1,N+1)}\geq 0,\;\;\;l\in \{1,2\},$  then obtain the desired equation (\ref{3-4}) i.e. $$\int _{a}^{b}\int _{c}^{d}P(y,z)\frac{(b-y)^{m} }{m!} \frac{(d-z)^{n} }{n!}dz\, dy=0\;\;\;0\leq m\leq M;\;\;\;0\leq n\leq N.$$
 In the similar manner, if take the following functions in (\ref{3-3})  $\forall\,\,\, s\in[a,b]$ and $0\leq n\leq N$
  $$
f^{3}(y,z)= \left\{ \begin{array}{rl}
\frac{(s-y)^M}{M!} \frac{(d-z)^n}{n!}, & y<s \\
0, & y\geq s
\end{array} \right.
\quad \text{and} \quad
f^{4}=-f^{3}
$$
such that $(-1)^{M+N}f^{l}_{(M+1,N+1)}\geq 0,\quad l\in \{3,4\}$, we obtain desired equation (\ref{3-5}) i.e. $$\int _{a}^{s}\int _{c}^{d}P(y,z)\frac{(s-y)^{M} }{M!} \frac{(d-z)^{n} }{n!}dz\, dy=0,\;\;\;0\leq n\leq N;\,\,\,\forall\,\,\,s\in[a,b].$$
Similarly, if take the following functions in (\ref{3-3}) $\forall\,\,\,t\in[c,d]$ and $0\leq m\leq M$
  $$
f^{5}(y,z)= \left\{ \begin{array}{rl}
\frac{(b-y)^m}{m!} \frac{(t-z)^N}{N!}, & z<t \\
0, & z\geq t
\end{array} \right.
\quad \text{and} \quad
f^{6}=-f^{5}
$$
such that $(-1)^{M+N}f^{l}_{(M+1,N+1)}\geq 0,\quad l\in \{5,6\}$, we can obtain above equation (\ref{3-6}) i.e. $$\int _{a}^{b}\int _{c}^{t}P(y,z)\frac{(b-y)^{m} }{m!} \frac{(t-z)^{N} }{N!}dz\, dy=0,\;\;\;0\leq m\leq M;\,\,\,\forall\,\,\,t\in[c,d].$$
By considering the below function in (\ref{3-3}), obtain the desired last inequality (\ref{3-7})  for $s\in[a,b],\;t\in[c,d]$
 $$
f^7(y,z)= \left\{ \begin{array}{rl}
\frac{(s-y)^{M} }{M!} \frac{(t-z)^{N} }{N!}, & y<s,\quad \quad z<t \\
0, & y\geq s\;\;\;\text{or}\;\;\;z\geq t.
\end{array} \right.
$$
\end{proof}
\begin{theorem}\label{positive-0}
Let suppositions of Theorem $\ref{th1}$ be true, then following inequality holds;
\begin{align} \label{3-31}
\Lambda(f)=\int_{a}^{b} \int_c^d P(y,z) f(y,z)dz\, dy\geq 0
\end{align}
for all completely monotonic function $f$ of order $(M+1,N+1)$ on $I\times J$ if
\begin{eqnarray} \label{3-41}
\int_a^b \int_c^d P(y,z) \frac{(b-y)^i}{i!}   \frac{(d-z)^j}{j!} dz \, dy &=&0, \,\,\, \,\,i\,\,\in \{0,\ldots,M\};\,\,\, j\in \{0,\ldots,N\} \\ \label{3-51}
\int_a^s \int_c^d P(y,z) \frac{(s-y)^M}{M!}\frac{(d-z)^j}{j!} dz \, dy&=&0,\,\,\,\,\,\, j\in \{0,\ldots,N\};\,\,\,\forall \,\,s\in[a,b] \\ \label{3-61}
\int_a^b \int_c^t P(y,z) \frac{(b-y)^i}{i!} \frac{(t-z)^N}{N!} dz \, dy&=&0,\,\,\, \,\,i\,\,\in \{0,\ldots,M\};\,\,\,\forall \,\,t\in[c,d] \\ \label{3-71}
\int_a^s \int_c^t P(y,z) \frac{(s-y)^M}{M!} \frac{(t-z)^N}{N!} dz \, dy&\geq& 0,\,\,\,\forall \,\,s\in[a,b];\,\,\,\forall \,\,t\in[c,d].
\end{eqnarray}
\end{theorem}
\begin{proof}
If (\ref{3-41}), (\ref{3-51}) and (\ref{3-61}) hold, then first, second and third sums are zero in (\ref{dodu}), then by using (\ref{3-71}) obtain the desired inequality (\ref{3-31}).
\end{proof}
\begin{remark}
If we simply put $f(y,z)=f(y)g(z)$ in previous two results, then we obtain the similar results for two functions $f$ and $g$.
\end{remark}
\section{Mean Value Theorems}
It is known fact that the mean value theorem is valueable tool for getting interesting and important results of classical real analysis. In the field of differential calculus, the most demanding theorems are Lagrange and Cauchy mean value theorems. For more materials on this topic (see \cite{LC-MVT}). Here, we would give some generalized mean value theorems of Lagrange and Cauchy-type.
\begin{theorem}\label{th3}
Let $P:I\times J \rightarrow \mathbb{R}$, be an integrable function and $f \in C^{(M+1,N+1)}(I\times J)$, be a $(M+1,N+1)-\nabla-$convex function on the interval $I\times J$. Let $\Lambda$ be a linear functional as stated in $(\ref{3-3})$ and the conditions $(\ref{3-4})$, $(\ref{3-5})$, $(\ref{3-6})$ and $(\ref{3-7})$ be true for function $P$ in Theorem  $\ref{positive}$, then  $\exists (\eta,\zeta) \,\,\in\,\, I\times J$ $\backepsilon$
\begin{equation}\label{3-8}
 \Lambda(f)= \Lambda(G_0) f_{(M+1,N+1)}(\eta,\zeta)
\end{equation}
where $\displaystyle G_0(y,z)=(-1)^{M+N}\frac{y^{M+1}}{(M+1)!}\frac{z^{N+1}}{(N+1)!}$.
\end{theorem}
\begin{proof}
Let$\quad U=\underset{(y,z)\in I\times J}{\max} (-1)^{M+N}f_{(M+1,N+1)}(y,z),\quad L=\underset{(y,z)\in I\times J}{\min}(-1)^{M+N}f_{(M+1,N+1)}(y,z).$ \\
Then the function$$G(y,z)=U(-1)^{M+N}\frac{y^{M+1}}{(M+1)!}\frac{z^{N+1}}{(N+1)!}-f(y,z)=U G_0(y,z)-f(y,z)$$
 gives us
$$(-1)^{M+N}G_{(M+1,N+1)}(y,z)=U-(-1)^{M+N}f_{(M+1,N+1)}(y,z)\geq0$$ i.e. $G$ is $\nabla-$convex function of order $(M+1,N+1)$ on $I\times J$.
Hence $\Lambda (G)\geq0$ using Theorem \ref{positive} and we would summerize that
$$\Lambda (f)\leq U \Lambda (G_0).$$
Similarly
$$L \Lambda(G_0)\leq \Lambda (f).$$
Now, we can write the above two inequalities as:
$$L \Lambda(G_0)\leq \Lambda (f)\leq U \Lambda (G_0)$$
which gives the required result (\ref{3-8}).
\end{proof}
\begin{theorem}\label{th2.9}
Let $f,g \in C^{(M+1,N+1)}(I\times J)$, be two $\nabla-$convex functions of order $(M+1,N+1)$ on the interval and $P:I\times J \to R$ be an integrable function. Let $\Lambda$ be a linear functional as stated in $(\ref{3-3})$ and the conditions $(\ref{3-4})$, $(\ref{3-5})$, $(\ref{3-6})$ and $(\ref{3-7})$ be true for function $P$ in Theorem  $\ref{positive}$, then $\exists$ $(\eta,\zeta) \,\,\in\,\, I\times J$ $\backepsilon$
\begin{equation*}
 \frac{\Lambda (f)}{\Lambda (g)}=\frac{f_{(M+1,N+1)}(\eta,\zeta)}{g_{(M+1,N+1)}(\eta,\zeta)}
\end{equation*}
considering with non-zero denominators.
\begin{proof}
Let $u\in C^{(M+1,N+1)}$ be a $\nabla-$convex function of order $(M+1,N+1)$ on the interval $I\times J$, be stated as:
$$u=\Lambda (g)f-\Lambda(f)g.$$
Applying Theorem \ref{th3}  $\exists$ $(\eta,\zeta)$ $\backepsilon$
$$0=\Lambda(u)=u_{(M+1,N+1)}(\eta,\zeta)\Lambda(G_0)$$
or
$$[\Lambda (g)f_{(M+1,N+1)}(\eta,\zeta)-\Lambda(f)g_{(M+1,N+1)}(\eta,\zeta)]\Lambda(G_0)=0$$
which gives desired result.
\end{proof}
\end{theorem}
\begin{corollary}
Let $\Lambda$ be a linear functional as stated in $(\ref{3-3})$ and the conditions $(\ref{3-4})$, $(\ref{3-5})$, $(\ref{3-6})$ and $(\ref{3-7})$ be true for function $P$ with $N=M$ in Theorem  $\ref{positive}$, then $\exists$ $(\eta,\zeta) \,\,\in\,\, I\times J$ $\backepsilon$
\begin{equation*}
{(\eta\zeta)}^{p-q}=\frac{[(q+1)q(q-1)\cdots(q-M+2)(q-M+1)]^2\Lambda((yz)^{p+1})}{[(p+1)p(p-1)\cdots(p-M+2)(p-M+1)]^2\Lambda((yz)^{q+1})}
\end{equation*}
where $p,q \not\in \{-1,0,1,\ldots,M-1\}$, but lie on $-\infty<p\neq q<+\infty$.
\end{corollary}
\begin{proof}
Let
\begin{equation*}
f(y,z)=(yz)^{p+1},\qquad g(y,z)=(yz)^{q+1}
\end{equation*}
Then
\begin{eqnarray*}
f_{(M+1,M+1)}=[(p+1)p(p-1)\ldots(p-M+2)(p-M+1)]^2 (yz)^{p-M} \\ \text{and} \qquad
g_{(M+1,M+1)}=[(q+1)q(q-1)\ldots(q-M+2)(q-M+1)]^2 (yz)^{q-M}
\end{eqnarray*}
\begin{eqnarray*}
\frac{f_{(M+1,M+1)}(\eta,\zeta)}{g_{(M+1,M+1)}(\eta,\zeta)} &&= \frac{[(p+1)p(p-1)\ldots(p-M+2)(p-M+1)]^2 (\eta\zeta)^{p-M}}{[(q+1)q(q-1)\ldots(q-M+2)(q-M+1)]^2 (\eta\zeta)^{q-M}} \\
&&= \frac{[(p+1)p(p-1)\ldots(p-M+2)(p-M+1)]^2}{[(q+1)q(q-1)\ldots(q-M+2)(q-M+1)]^2 } (\eta\zeta)^{p-q}.
\end{eqnarray*}
On the other hand:
\begin{eqnarray*}
\Lambda(f)=\Lambda((yz)^{p+1})
\end{eqnarray*}
So if we put all these in the equality Theorem \ref{th2.9}, we get
\begin{eqnarray*}
\frac{\Lambda(f)}{\Lambda(g)} &&= \frac{f_{(M+1,M+1)}(\eta,\zeta)}{g_{(M+1,M+1)}(\eta,\zeta)} \\
\frac{\Lambda((yz)^{p+1})}{\Lambda((yz)^{q+1})} &&= \frac{[(p+1)p(p-1)\ldots(p-M+2)(p-M+1)]^2}{[(q+1)q(q-1)\ldots(q-M+2)(q-M+1)]^2 } (\eta\zeta)^{p-q} \\
(\eta\zeta)^{p-q} &&= \frac{[(q+1)q(q-1)\ldots(q-M+2)(q-M+1)]^2 \Lambda((yz)^{p+1})}{[(p+1)p(p-1)\ldots(p-M+2)(p-M+1)]^2 \Lambda((yz)^{q+1})}
\end{eqnarray*}
\end{proof}
\begin{remark}
Here we observe that for the case $M=N$ the $(M+1,N+1)-\nabla-$convex function becomes $(M+1,M+1)-$convex function and hence we retrieve the results from \cite{Asif-popoviciu}.
\end{remark}
\section {\bf{Exponential Convexity}}
Here, In this section $J=(a,b)\in\mathbb{R}$.
\begin{definition}\cite{B}
A function $\omega:J\rightarrow\mathbb{R}$ is exponentially convex on open interval $J$, if $\omega$ is continuous and
\begin{eqnarray*}
\sum_{i,j=1}^{m}\rho_{i}\rho_{j}\,\omega\left(y_{i}\,+\,y_{j}\right)\,\geq\,0
\end{eqnarray*}
 $\forall \,m\in\mathbb{N}$ and $\forall$ \,$\rho_{i}, \rho_{j}\in\mathbb{R}$; such that
$y_{i}+y_{j}\in J$ and $i,j\in\{1,\ldots,m\}$.
\end{definition}
\begin{example}\label{ex1}
A function $y\mapsto ce^{ky}$, example of exponentially convex for $k \in \mathbb{R}$ and constant $c\geq0$.
\end{example}
\begin{proposition}\label{prop}\cite{mnpj}
Let $\omega:J\rightarrow\mathbb{R}$, then given statements are similar:
\begin{enumerate}
  \item [(i)] $\omega$ is exponentially convex on open interval $J$.
  \item [(ii)]$\omega$ is continuous and
\begin{eqnarray*}
\sum_{i,j=1}^{m}\rho_{i}\rho_{j}\,\omega\left(\frac{y_{i}\,+\,y_{j}}{2}\right)\,\geq\,0,
\end{eqnarray*}
$\forall \,\rho_{i}, \rho_{j} \in\mathbb{R}$ and every $y_{i}, y_{j}\in J$; $i,j\in\{1,2,3,\ldots,m-1,m\}$.
\end{enumerate}
\end{proposition}
\begin{corollary}\label{corr-3}
If  $\omega$ is exponentially convex function on open interval $J$, then the following matrix is a positive semi-definite matrix. $$\left[\omega\left(\frac{y_{i}\,+\,y_{j}}{2}\right)\right]_{i,j=1}^{m}$$  Particularly, i.e.
    \begin{equation*}
det\left[\omega\left(\frac{y_{i}\,+\,y_{j}}{2}\right)\right]_{i,j=1}^{m}\,\geq\,0,
\end{equation*}
$\forall \,m\in\mathbb{N}, \,y_{i}, y_{j}\in J$; $i,j\in\{1,2,3,\cdots,m-1,m\}$.
  \end{corollary}
\begin{corollary}\label{corr}
If $\omega:J\rightarrow \mathbb{R_{+}}$ is an exponentially convex function. Then functin $\omega$ is a $\log$-convex function, where $\forall y,z \in J$ and  $\forall \lambda \in [0,1]$, we have
$$\omega(\lambda y+(1-\lambda)z)\leq {\omega}^{\lambda}(y) {\omega}^{1-\lambda}(z).$$
\end{corollary}
Let$\quad D=\{\psi^{(q)}:(0,\infty)\times(0,\infty)\rightarrow\mathbb{R}| q\in \mathbb{R}\}$ be a family of functions stated as:
\begin{equation*}\label{a2}
\psi^{(q)}(y,z)\,=\,\left\{
                \begin{array}{rl}
                   \displaystyle \frac{(y+k_{1})^{q}(z-k_{2})^{q}}{[q(q-1)\cdots(q-M)]^2},&\quad \;q\not\in \{0,1,\ldots,M\}\\
                  \displaystyle  \frac{ (y+k_{1})^{q}(z-k_{2})^{q} [\log (y+k_{1})(z-k_{2})]^2}{2[q!(M-q)!]^2},&\quad \;q \in \{0,1,\ldots,M\}
                \end{array}
              \right.
\end{equation*}
Clearly $\psi^{(q)}_{(M+1,M+1)}(y,z)=[(y+k_{1})(z-k_{2})]^{q-M-1}=e^{(q-M-1)\log[(y+k_{1})(z-k_{2})]}$ for $(y+k_{1},z-k_{2})\in (0,\infty)\times(0,\infty)$ so $\psi^{(q)}$ is completely monotonic function of order $(M+1,M+1)$ and $q\mapsto \psi^{(q)}_{(M+1,M+1)}$ is an exponentially convex function $f$ on real numbers. We can say that every positive exponentially convex  is log$-$convex function, by using the above mention Corollary \ref{corr} \\
Now, at this stage we can give next theorem which is stated as:
\begin{theorem}\label{nexpthem}
Let $\Lambda$ be a linear functional as stated in $(\ref{3-3})$ and the conditions $(\ref{3-4})$, $(\ref{3-5})$, $(\ref{3-6})$ and $(\ref{3-7})$ be true for function $P$ in Theorem  $\ref{positive}$ and $\psi^{(q)}$ be  a completely monotonic function defined in previous, then following points hold:
\begin{itemize}
\item[(i)] $q\mapsto \Lambda(\psi^{(q)})$ is continuous on $\mathbb{R}$.
\item[(ii)] $q\mapsto \Lambda(\psi^{(q)})$ is exponentially convex function on $\mathbb{R}$.
\item[(iii)] If $q\mapsto\Lambda(\psi^{(q)})$ is positive function on $\mathbb{R}$, then the $q\mapsto\Lambda(\psi^{(q)})$ is log-convex on $\mathbb{R}$. Moreover, the following inequality holds for $r<s<t; \,r,\,s,\,t\,\in J$
\begin{equation}\label{lypo}
[\Lambda_k(f_s)]^{t-r}\, \le\, [\Lambda_k(f_r)]^{t-s}\,[\Lambda_k(f_t)]^{s-r}
\end{equation}
\item[(iv)] For every $m\in\mathbb{N}$ and $q_{1},\ldots,q_{m}\in \mathbb{R}$, the following matrix is positive semi-definite.
\begin{equation*}
\left[\Lambda(\psi^{(\frac{q_{i}+q_{j}}{2})})\right]_{i,j=1}^{m}
\end{equation*}
Particularly,
\begin{equation*}\label{eq5}
\mathrm{det}\left[\Lambda(\psi^{(\frac{q_{i}+q_{j}}{2})})\right]_{i,j=1}^{m}\,\geq\,0 
\end{equation*}
\item[(v)]  If $q\mapsto\Lambda(\psi^{(q)})$ is differentiable on $\mathbb{R}$, and $\forall s,t,u,v\in \mathbb{R}$, $\backepsilon$ $t\leq v$ and  $s\leq u$, then we have
        \begin{equation}\label{3-9}
      \mathfrak{ M}_{s,t}(y,z)\leq \mathfrak{ M}_{u,v}(y,z)
        \end{equation}
  where
      \begin{equation*}
        \mathfrak{ M}_{s,t}(y,z)=\begin{cases}
        \displaystyle \left(\frac{\Lambda(\psi^{(s)})}{\Lambda(\psi^{\{t\}})}\right)^{\frac{1}{s-t}}\quad ,\,\,\,s\neq t\\
        \displaystyle  \exp\left(\frac{\frac{d}{ds}\Lambda(\psi^{(s)})}{\Lambda(\psi^{(s)})}\right),\,\,s=t
        \end{cases}
      \end{equation*}
     for $\psi^{(s)},\psi^{(t)}\in D$.
 \end{itemize}
\end{theorem}
\begin{proof}
(i) For fixed $M \in \mathbb{N} \cup \{0\}$,
using L' H$\hat{{\rm o}}$pital rule two times and then apply limit, we obtain
\begin{eqnarray*}
\lim\limits_{q \rightarrow 0 }\Lambda(\psi^{(q)})&=&\lim\limits_{q \rightarrow 0 } \displaystyle  \frac{\int _{a}^{b}\int _{a}^{b}P(y,z)(y+k_{1})^{q}(z-k_{2})^{q}dzdy}{[q(q-1)\cdots(q-M)]^2}\\&=&\displaystyle  \frac{\int _{a}^{b}\int _{a}^{b}P(y,z)[\log (y+k_{1})(z-k_{2})]^2dzdy}{2{[M!]}^2}\\
&=&\Lambda(\psi^{(0)}).
\end{eqnarray*}
Similarly, we can show $$\lim\limits_{q \rightarrow k }\Lambda(\psi^{(q)})=\Lambda(\psi^{(k)},)\quad k\in \{1,2,\ldots,M\}.$$
(ii) Let us define the function
$$\eta(y,z)=\sum_{i,j=1}^{k}\alpha_{i}\alpha_{j}\psi^{(\frac{q_i+q_j}{2})}(y,z),$$
$q_{i}\in \mathbb{R} ,\alpha_i\in\mathbb{R}$ and $i\in\{1,2,\ldots,k\}.$\\
Since the function $q\mapsto \psi^{(q)}_{(M+1,M+1)}$ is exponentially convex function, we write
$$\eta_{(M+1,M+1)}=\sum_{i,j=1}^{k}\alpha_{i}\alpha_{j}\psi^{(\frac{q_i+q_j}{2})}_{(M+1,M+1)}\geq0,$$ which implies that $\eta$ is $\nabla-$convex function of order $(M+1,M+1)$ on $\mathbb{R_{+}}\times\mathbb{R_{+}}$ and we have
$\Lambda(\eta)\geq0$. Hence
$$\sum_{i,j=1}^{k}\alpha_{i}\alpha_{j}\Lambda(\psi^{(\frac{q_i+q_j}{2})})\geq0.$$
On the behalf of above working we can summarize that function $q\rightarrow \Lambda(\psi^{(q)})$ is  exponentially convex on real numbers.\\
\indent(iii)    It follows from $(ii)$ and Corollary $\ref{corr}$. As the function $t\mapsto \Lambda (f_t)$ is log-convex i.e. $\ln\Lambda(f_t)$ is convex. Now using the definition of convex functions from \cite[p.~2]{redbook}, we have
\begin{equation*}\label{def2-convex-ineq}
\left(y_3-y_2\right)f(y_1)+\left(y_1-y_3\right)f(y_2)+
\left(y_2-y_1\right)f(y_3)\geq 0
\end{equation*}
holds for each $y_1,\,y_2,\,y_3 \in I$ such that $y_1 < y_2 < y_3$,
which gives (\ref{lypo}), i.e.,
\begin{equation*}
    \ln[\Lambda(f_{s})]^{t-r}\,\leq\,
    \ln[\Lambda(f_{r})]^{t-s}\,+\,\ln[\Lambda(f_{t})]^{s-r},
\end{equation*}
\indent (iv)    It is consequence of Corollary \ref{corr-3}.\\
\indent (v)   We recall an another definition of convex function $\varphi $ from \cite [p.2]{redbook}
\begin{equation}\label{3-10}
  \frac{\varphi\left(s\right)\,-\,\varphi\left(t\right)}{s\,-\,t}\,\leq\, \frac{\varphi\left(u\right)\,-\,\varphi\left(v\right)}{u\,-\,v},
\end{equation}
$ \forall \, s, t,u,v \in J$ $\backepsilon$ $s\leq u,\,t\leq v,\,s\neq t,\,u\neq v$. \\Since by (iii), $ \Lambda(\psi^{(q)}) $ is $\log$-convex, so by setting $\varphi(y)=\log\Lambda(\psi^{(y)})$ in (\ref{3-10}) we have
\begin{equation}\label{3-11}
\frac{\log\Lambda(\psi^{(s)})\,-\log\Lambda(\psi^{\{t\}})}{s-t} \,\leq\,
\frac{\log\Lambda(\psi^{(u)})-\log\Lambda(\psi^{(v)})}{u-v}
\end{equation}
for  $s\leq u,\,t\leq v,\,s\neq t,\,u\neq v$, which is similar to (\ref{3-9}).
The cases for $s= t$ and$/$or $u= v$ are simply getting from (\ref{3-11}) by using the respective limits.
\end{proof}

\section{Examples with Applications}
At last topic of our paper, construct different examples of completely monotonic, exponentially convex functions and applications by using different classes of functions $F=\{f^{q}:q \in I\subset\mathbb{R}\}$. Let us consider the following examples

\begin{example}
 Let $F_{1}= \{\zeta^{q}:\mathbb{R_{+}}\times \mathbb{R_{+}}\rightarrow
\mathbb{R_{*}} | q\in\mathbb{R}\}$, be a family of functions which is stated as
\[
\label{l2.102}\zeta^{q}(y,z)\,=\,\left\{
\begin{array}
[c]{ll}%
\frac{e^ {q(y+z)}}{q^{2m+2}} & ,\quad\hbox{q\,$\neq$ 0}\\
\frac{(y+z)^{2m+2}}{(2m+2)!} & ,\quad\hbox{q\,= 0}
\end{array}
\right.
\]
Since $\zeta^{q}_{(M+1,M+1)}(y,z)=e^{q(y+z)}>0$, the function $\zeta^{q}(y,z)$
is a $(M+1,M+1)-$completely monotonic on $\mathbb{R}$, $\forall$ $q\in\mathbb{R}$ and $q\to
\zeta^{q}_{(M+1,M+1)}(y,z)$ is exponentially convex by definition.
\end{example}

\begin{example}
Let $F_{3}=\{\phi^{q}:\mathbb{R_{+}}\times\mathbb{R_{+}}\rightarrow\mathbb{R_{+}} | q\in\mathbb{R_{+}}\}$, be a family of functions which is stated as
\begin{eqnarray*}
\phi^{q}(y,z)= \left\{
\begin{array}
[c]{ll}
\frac{(yz)^{q}}{[q(q-1)\cdots(q-m)]^{2}} & ,\quad q\not \in {\{0,1,2,\ldots,m-1,m\}}\\
\frac{(yz)^{q}\ln(yz)^{2}}{2[q!(m-q)!]^{2}} & ,\quad q\in {\{0,1,2,\ldots,m-1,m\}}
\end{array}
\right.
\end{eqnarray*}
Since $\phi^{q}_{(M+1,M+1)}(y,z)= e^{(q-m-1)\ln(yz)}>0$, the function $\phi^{q}(y,z)$
is a $(M+1,M+1)-$completely monotonic on $\mathbb{R}$, $\forall$ $q\in\mathbb{R}$ and $q\to
\phi^{q}(y,z)_{(M+1,M+1)}$ is exponentially convex by definition.
\end{example}
In all the above examples we can use similar arguments as in proof of Theorem $\ref{nexpthem}$.

\end{document}